\documentclass[11pt,leqno]{amsart}

\usepackage{amsmath}
\usepackage{amssymb}
\usepackage{a4wide}
\usepackage{bbm}
\usepackage{graphics}
\usepackage{epsfig}
\usepackage{color}
\usepackage{mathrsfs}
\usepackage[sans]{dsfont}
\newtheorem{theorem}{Theorem}[section]
\newtheorem{lemma}[theorem]{Lemma}
\newtheorem{proposition}[theorem]{Proposition}

\newtheorem{remark}[theorem]{Remark}
\newtheorem{corollary}[theorem]{Corollary}

\newcounter{hypo}
\newcounter{hypoa}

\newcounter{hypoaa}
\newcounter{hypobb}

\def\C{{\mathbb C}}

\def\N{{\mathbb N}} 
\def\R{{\mathbb R}}

\def\CD{\mathcal {D}}

\def\CO{\mathcal {O}}
\def\CP{\mathcal {P}}

\def\CS{\mathcal {S}}

\def\SB{\mathscr {B}}

\def\SD{\mathscr {D}}

\def\ker{\mathop{\rm Ker}\nolimits}

\newcommand{\nequiv}{\equiv \!\!\!\!\!\!\;\! / \,\,}

\def\one{\mathds{1}}
\def\re{\mathop{\rm Re}\nolimits}
 \def\im{\mathop{\rm Im}\nolimits}

\renewcommand{\div}{\operatorname{div}}

\newcommand{\Vect}{\operatorname{Vect}}

\newcommand{\supp}{\operatorname{supp}}

\def\dist{\mathop{\rm dist}\nolimits}

\def\hess{\mathop{\rm Hess}\nolimits}

\def\<{\langle}
\def\>{\rangle}

\makeatletter
 \@addtoreset{equation}{section}
 \makeatother
 
\title{Real diffusion with complex spectral gap}

\author[J.-F. Bony]{Jean-Fran\c{c}ois Bony}
\address{Jean-Fran\c{c}ois Bony, IMB, CNRS (UMR 5251), Universit\'e de Bordeaux, 33405 Talence, France}
\email{bony@math.u-bordeaux.fr}
\author[L. Michel]{Laurent Michel}
\address{Laurent Michel, IMB, CNRS (UMR 5251), Universit\'e de Bordeaux, 33405 Talence, France}
\email{lamichel@math.u-bordeaux.fr}

\subjclass[2010]{}

\begin{document}

\begin{abstract}
The low-lying eigenvalues of the generator of a Langevin process are known to satisfy the Eyring--Kramers law in the low temperature regime under suitable assumptions. These eigenvalues are generically real. We construct generators whose spectral gap is given by non-real eigenvalues or by a real eigenvalue having a Jordan block.
\end{abstract}

\maketitle

\section{Introduction} \label{s0}

The generator of a diffusion process is generally a differential operator of order two with real coefficients. In the last decades, the asymptotic of its low-lying eigenvalues has been obtained \cite{BoLeMi22_01,BoGaKl05_01,DiLeLeNe19_01,HeKlNi04_01,HeHiSj08-2,HeHiSj11_01,LaMaSe19,LePMi20} in the low temperature regime (Eyring--Kramers law), see \cite{Be13} for a general presentation. These results provide sharp informations on metastability or on return to equilibrium. For reversible processes, the generator is a self-adjoint operator on an appropriate Hilbert space and then its spectrum is always real. For irreversible processes, the generator is no longer self-adjoint on the natural Hilbert space and one can hope to observe non-real eigenvalues or Jordan's blocks. But, as recalled at the end of this part, there are strong constrains on the low-lying spectrum of generators which make such phenomena unlikely and explain why non-real spectra have not been obtained up to now. The goal of this paper is to construct generators with pathologic spectral gap.

We first discuss spectral properties of generators in the general setting of \cite{BoLeMi22_01} and send the reader to this paper for precise statements and to the references of the previous paragraph for slightly different settings. In \cite{BoLeMi22_01}, we consider the operator on $L^{2} ( \R^{d} )$
\begin{equation} \label{a80}
P = - h \div\circ A \circ h \nabla + \frac{1}{2} \big( b \cdot h \nabla + h \div \circ b \big) + c ,
\end{equation}
where the symmetric matrix $A = ( a_{j , k} ( x , h ) )_{j , k}$, the vector field $b = ( b_{j} ( x , h ) )_{j}$ and the function $c ( x , h )$ are smooth and real-valued. Moreover, these functions are symbols and have an asymptotic expansion in power of the parameter $h$, which is proportional to the temperature. We assume that $P$ has an invariant distribution which has a Gibbs form. More precisely, there exists a confining smooth Morse function $f$ such that
\begin{equation*}
P ( e^{- f / h} ) = P^{*} ( e^{- f /h} ) = 0 .
\end{equation*}
Let $1 \leq n_{0} < + \infty$ denote the number of minima of $f$. Hypoelliptic and hypocoercive assumptions are also made. Under these assumptions, $P$ is maximal accretive and has domain
\begin{equation*}
\CD ( P ) = \{ u \in L^{2} ( \R^{d} ) ; \, P u \in L^{2} ( \R^{d} ) \} ,
\end{equation*}
as proved in Section 3 of \cite{HeHiSj08-2}. The evolution equation naturally associated to $P$ is the heat (or Fokker--Planck) equation
\begin{equation} \label{a37}
\left\{ \begin{aligned}
&h \partial_{t} u ( t , x ) = - P u ( t , x ) , \\
&u ( 0 , x ) = u_{0} ( x ) ,
\end{aligned} \right.
\end{equation}
where $u_{0} ( x ) \in L^{2}( \R^{2} )$ is the initial data. The low-lying spectrum of $P$ is given by the following result (see Theorem 3 of \cite{BoLeMi22_01}).

\begin{theorem}[Eyring--Kramers law]\sl \label{a81}
There exists $\lambda_{*} > 0$ such that, for $h$ small enough, $P$ has exactly $n_{0}$ eigenvalues counted with their algebraic multiplicity $\lambda_{1} ( h ) , \ldots , \lambda_{n_{0}} ( h )$ in $\{ z \in \C ; \ \re z \leq \lambda_{*} h \}$. Moreover, $\lambda_{1} ( h ) = 0$ is simple with $\ker P = e^{- f / h} \C$. For $n = 2 , \ldots , n_{0}$, the eigenvalue $\lambda_{n} ( h )$ satisfies the asymptotic
\begin{equation*}
\lambda_{n} ( h ) = a_{n} ( h ) h e^{- 2 S_{n} / h} \qquad \text{with} \qquad a_{n} ( h ) \simeq \sum_{j \geq 0} a_{n}^{j} h^{j} ,
\end{equation*}
$S_{n} = f ( s_{n} ) - f ( m_{n} ) > 0$ for some particular saddle point $s_{n}$ and minimum $m_{n}$, $a_{n}^{0} \neq 0$ explicitly known and $a_{n}^{j} \in \R$ for all $j \neq 0$.
\end{theorem}

Note that the first eigenvalue $\lambda_{1} = 0$ is always real. Since all the coefficients $a_{n}^{j}$ are real, it is not possible to use the Eyring--Kramers law to construct an operator with non-real small eigenvalues. Moreover, the imaginary part of $\lambda_{n}$ is always extremely small. More precisely,

\begin{remark}\sl \label{a84}
For all $n = 1 , \ldots , n_{0}$, we have
\begin{equation} \label{a85}
\vert \im \lambda_{n} \vert = \CO ( h^{\infty} ) \re \lambda_{n} .
\end{equation}
\end{remark}

On the other hand, the particular form \eqref{a80} of the generator $P$ induces symmetries on its spectrum, as remarked on page 15 of \cite{LePMi20}. More precisely, since the coefficients of $P$ are real-valued and the domain of $P$ is stable by complex conjugation, we get
\begin{equation} \label{a82}
\overline{( P - \lambda ) u} = ( P - \overline{\lambda} ) \overline{u} ,
\end{equation}
for all $\lambda \in \C$ and $u \in \CD ( P )$. This implies the following property which is also satisfied for $\mathcal{PT}$-symmetric operators (see for instance \cite{BeBoMeRaSj16_01} for the bifurcation of eigenvalues from the real axis to the complex plane).

\begin{remark}\sl  \label{a83}
The spectrum of $P$ is invariant by complex conjugation.
\end{remark}

In particular, when $f$ has exactly two minima, $P$ has two small eigenvalues $\lambda_{1} = 0$ and $\lambda_{2}$ by Theorem \ref{a81}. Since $\lambda_{2} = \overline{\lambda_{2}}$ by Remark \ref{a83}, these two small eigenvalues are always real and simple for $h$ small enough (see Remark 1.10 of \cite{LePMi20}).

More generally, if the asymptotic expansion of $\lambda_{n}$ given by the Eyring--Kramers law is different from that of the other eigenvalues, then $\lambda_{n}$ is real and simple for $h$ small enough. As an example, if the (Arrhenius) exponential factors $S_{n}$ are all different, then all the small eigenvalues $\lambda_{n}$ are real and simple for $h$ small enough. This shows that the exponentially small eigenvalues of the generator of a diffusion as in \eqref{a80} are generically real.

We now construct operators of the form \eqref{a80} with non-real small eigenvalues or Jordan blocks. From the two previous paragraphs, the associated Morse function $f$ must have at least $3$ minima and some of exponential factors $S_{n} = f ( s_{n} ) - f ( m_{n} )$ must coincide.

\section{Statement of the results} \label{s1}


\begin{figure}
\begin{center}
\begin{picture}(0,0)%
\includegraphics{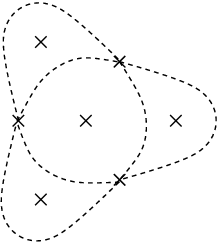}%
\end{picture}%
\setlength{\unitlength}{1184sp}%
\begingroup\makeatletter\ifx\SetFigFont\undefined%
\gdef\SetFigFont#1#2#3#4#5{%
  \reset@font\fontsize{#1}{#2pt}%
  \fontfamily{#3}\fontseries{#4}\fontshape{#5}%
  \selectfont}%
\fi\endgroup%
\begin{picture}(5795,6411)(1314,-7196)
\put(6076,-4486){\makebox(0,0)[b]{\smash{{\SetFigFont{9}{10.8}{\rmdefault}{\mddefault}{\updefault}$m_{1}$}}}}
\put(2476,-2386){\makebox(0,0)[b]{\smash{{\SetFigFont{9}{10.8}{\rmdefault}{\mddefault}{\updefault}$m_{3}$}}}}
\put(2476,-6586){\makebox(0,0)[b]{\smash{{\SetFigFont{9}{10.8}{\rmdefault}{\mddefault}{\updefault}$m_{2}$}}}}
\put(4426,-5161){\makebox(0,0)[b]{\smash{{\SetFigFont{9}{10.8}{\rmdefault}{\mddefault}{\updefault}$s_{1}$}}}}
\put(2251,-4111){\makebox(0,0)[b]{\smash{{\SetFigFont{9}{10.8}{\rmdefault}{\mddefault}{\updefault}$s_{2}$}}}}
\put(4426,-2911){\makebox(0,0)[b]{\smash{{\SetFigFont{9}{10.8}{\rmdefault}{\mddefault}{\updefault}$s_{3}$}}}}
\put(3601,-4636){\makebox(0,0)[b]{\smash{{\SetFigFont{9}{10.8}{\rmdefault}{\mddefault}{\updefault}$M$}}}}
\end{picture}%
$\qquad \qquad$ \includegraphics[width=110pt]{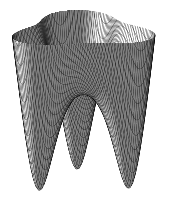}
\end{center}
\caption{The structure of the critical points of $f$ and an example of such a Morse function.} \label{f1}
\end{figure}
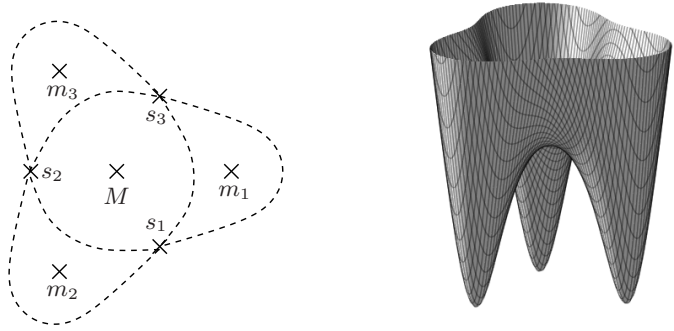

On $\R^{2}$, we consider a smooth Morse function $f$ with $f ( x ) = x^{2}$ outside a compact set and which is invariant under $R$, the rotation of angle $2 \pi / 3$ around $0$. Moreover, we assume that the set of critical points of $f$ consists of $3$ (global) minima $m_{1} , m_{2} , m_{3}$, $3$ saddle points $s_{1} , s_{2} , s_{3}$ and $1$ (local) maximum $M$ as in Figure \ref{f1}. Let $P_{0}$ be the Witten Laplacian associated to the function $f$, that is
\begin{equation} \label{k1}
P_{0} = d_{f}^{*} \circ d_{f} \qquad \text{with} \qquad d_{f} = e^{- f / h} \circ h \nabla \circ e^{f / h} = \begin{pmatrix}
h \partial_{x_{1}} + \partial_{x_{1}} f \\
h \partial_{x_{2}} + \partial_{x_{2}} f
\end{pmatrix} .
\end{equation}
A classical computation shows that this operator has the form
\begin{equation*}
P_{0} = - h^{2} \Delta + \vert \nabla f \vert^{2} - h \Delta f .
\end{equation*}
Since $f$ is a compactly supported perturbation of $x^{2}$, $P_{0}$ is self-adjoint on the domain of the harmonic oscillator $\SD ( P_{0} ) = H^{2} ( \R^{2} ) \cap \< x \>^{- 2} L^{2} ( \R^{2} )$, has a compact resolvent, $P_{0} \geq 0$ and
\begin{equation*}
\ker P_{0} = e^{- f / h} \C .
\end{equation*}
The spectrum of Witten Laplacians in such a geometric configuration has been studied in \cite[Section 7.4]{HeHiSj11_01}, \cite[Section 7C3]{Mi19_01} and \cite[Section 9.3]{LeNiVi20_01}. We send the reader to \cite{HeSj85_01} or to the second edition of the book \cite{CyFrKiSi87_01} for details on Witten Laplacians.

Throughout the paper, we set $S = f ( s ) - f ( m ) > 0$ and $\mu ( s ) < 0$ denotes the unique negative eigenvalue of $\hess f ( s )$. Since $f$ is invariant by rotation, these quantities do not depend on the minimum $m$ and the saddle point $s$ where they are computed. The bottom of the spectrum of $P_{0}$ is given by the following result.

\begin{proposition}[low eigenvalues of $P_{0}$]\sl \label{a30}
There exists $\lambda_{*} > 0$ such that, for $h$ small enough, $P_{0}$ has exactly three eigenvalues counted with multiplicity $\lambda_{1} ( h ) , \lambda_{2} ( h ) , \lambda_{3} ( h )$ in $] - \infty , \lambda_{*} h ]$. Moreover,
\begin{equation*}
\lambda_{1} = 0 , \qquad \lambda_{2} = \lambda_{3} \qquad \text{and} \qquad \lambda_{2} \sim \frac{3 \vert \mu ( s ) \vert \vert \det \hess f ( m ) \vert^{1 / 2}}{\pi \vert \det \hess f ( s ) \vert^{1 / 2}} h e^{- 2 S / h} .
\end{equation*}
\end{proposition}

This Proposition is mainly a consequence of previous results (see \cite{HeHiSj11_01,Mi19_01}). The unique novelty is that $\lambda_{2}$ has multiplicity two. This point and other spectral properties of $P_{0}$ are proved in Section \ref{s2}. Since $f$ is invariant under $R$, so are $P_{0}$ and all its eigenspaces.

We now construct an operator having a non-real spectral gap. For that, we perturb the operator $P_{0}$ by an anti-adjoint differential operator of order one. More precisely, we consider the operator
\begin{equation} \label{a77}
P_{\rm com} = P_{0} + B \qquad \text{with} \qquad B = \frac{1}{2} \big( b \cdot h \nabla + h \div \circ b \big) .
\end{equation}
We require that the vector-valued function $b ( x , h ) \in C^{\infty}_{0} ( \R^{2} ; \R^{2} )$ is a compactly supported real symbol of class $S ( h^{\infty} )$, where
\begin{equation*}
S ( r ) = \big\{ b ( x, h ) \in C^{\infty} ( \R^{2} ) ; \ \forall \alpha \in \N^{2} , \ \exists C_{\alpha} > 0 , \ \forall x \in \R^{2} , \ \forall h \in ] 0 , 1 ] , \ \vert \partial_{x}^{\alpha} b ( x , h ) \vert \leq C_{\alpha} r ( h ) \} ,
\end{equation*}
and $b \in S ( h^{\infty} )$ means that $b \in S ( h^{j} )$ for all $j \in \N$. In particular, $P_{\rm com}$ is closed on the domain $\SD ( P_{0} )$. We also assume that
\begin{equation} \label{a31}
B ( e^{- f / h} ) = 0.
\end{equation}
Then, the operator $P_{\rm com}$ enters into the setting of \eqref{a80}.

\begin{figure}
\begin{center}
\begin{picture}(0,0)%
\includegraphics{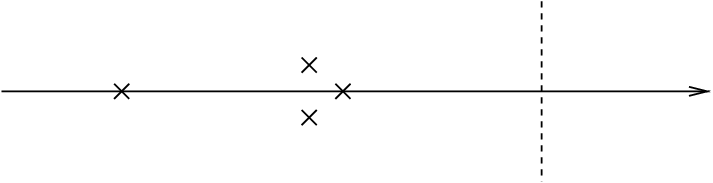}%
\end{picture}%
\setlength{\unitlength}{1579sp}%
\begingroup\makeatletter\ifx\SetFigFont\undefined%
\gdef\SetFigFont#1#2#3#4#5{%
  \reset@font\fontsize{#1}{#2pt}%
  \fontfamily{#3}\fontseries{#4}\fontshape{#5}%
  \selectfont}%
\fi\endgroup%
\begin{picture}(14241,3666)(-2432,-994)
\put(  1,239){\makebox(0,0)[b]{\smash{{\SetFigFont{9}{10.8}{\rmdefault}{\mddefault}{\updefault}$\mu_{1} = \lambda_{1} = 0$}}}}
\put(3751,1889){\makebox(0,0)[b]{\smash{{\SetFigFont{9}{10.8}{\rmdefault}{\mddefault}{\updefault}$\mu_{2}$}}}}
\put(3751,-286){\makebox(0,0)[b]{\smash{{\SetFigFont{9}{10.8}{\rmdefault}{\mddefault}{\updefault}$\mu_{3}$}}}}
\put(4876,239){\makebox(0,0)[b]{\smash{{\SetFigFont{9}{10.8}{\rmdefault}{\mddefault}{\updefault}$\lambda_{2} = \lambda_{3}$}}}}
\put(8626,1064){\makebox(0,0)[lb]{\smash{{\SetFigFont{9}{10.8}{\rmdefault}{\mddefault}{\updefault}$\lambda_{*} h / 2$}}}}
\put(11551,1064){\makebox(0,0)[lb]{\smash{{\SetFigFont{9}{10.8}{\rmdefault}{\mddefault}{\updefault}$\R$}}}}
\end{picture}%
\end{center}
\caption{The low-lying eigenvalues of $P_{0}$ and $P_{\rm com}$.} \label{f2}
\end{figure}
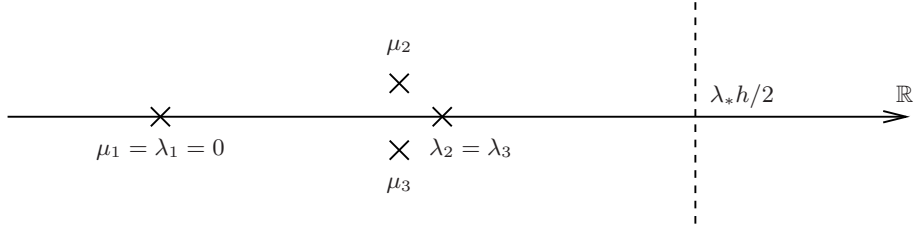

\begin{theorem}[Non-real eigenvalues]\sl \label{a20}
Let $r ( h ) = \CO ( h^{\infty} )$ be a positive function. There exists a function $b ( x , h ) \in C^{\infty}_{0}( \R^{2} ; \R^{2} ) \cap S ( r )$ with \eqref{a31} such that the spectrum of $P_{\rm com}$ satisfies
\begin{equation*}
\sigma ( P_{\rm com} ) \cap \{ z \in \C ; \ \re z < \lambda_{*} h / 2 \} = \big\{ \mu_{1} ( h ) , \mu_{2} ( h ) , \mu_{3} ( h ) \big\} ,
\end{equation*}
for $h$ small enough, with $\mu_{1} = 0$, $\mu_{2} = \lambda_{2} + \CO ( r )$, $\mu_{3}= \overline{\mu_{2}}$ and
\begin{equation}
\im \mu_{2} \neq 0 .
\end{equation}
\end{theorem}

Here and in the sequel, $\sigma ( T )$ denotes the spectrum of the operator $T$ and the eigenvalues $\mu_{\bullet} ( h )$ are simple for $h$ small enough. The setting of Theorem \ref{a20} is illustrated in Figure \ref{f2}. The symbol $b ( x , h )$ is only partially explicit (see Lemma \ref{a2}, \eqref{a24} and \eqref{a28}). In particular, its size may be way more smaller than $r$. Then, the imaginary part of $\mu_{2}$ and $\mu_{3}$ is very small. But, as explained in \eqref{a85}, it is always the case in the general setting.

Theorem \ref{a20} is proved using the perturbation theory at fixed $h$ small enough. In particular, its proof shows that operators as in \eqref{a77} with a small enough anti-adjoint part $B$ have a non-real spectral gap as soon as the leading term coming from the perturbation theory does not vanish (see Lemma \ref{a22}). In this sense, the situation of Theorem \ref{a20} is generic.

For $h$ small enough, let $\Pi_{\mu_{j}}$ denote the spectral projection of $P_{\rm com}$ associated to the eigenvalue $\mu_{j}$. Using the Cauchy formula, it can be written
\begin{equation*}
\Pi_{\mu_{j}} = \frac{1}{2 i \pi} \oint_{\gamma} ( z - P_{\rm com} )^{- 1} d z ,
\end{equation*}
where $\gamma$ is a sufficiently small loop around $\mu_{j}$ positively oriented. The relations \eqref{a81} and \eqref{a31} give
\begin{equation} \label{a38}
\Pi_{\mu_{1}} = \frac{e^{- f ( x  ) / h}}{\Vert e^{- f / h} \Vert^{2}} \big\< e^{- f / h} , \cdot \big\> \qquad \text{and} \qquad \overline{\Pi_{\mu_{2}} u} = \Pi_{\mu_{3}} \overline{u} .
\end{equation}
Let $u ( x , h )$ be an eigenvector of $P_{\rm com}$ associated to the eigenvalue $\mu_{2}$. From \eqref{a38}, $\overline{u}$ is an eigenvector associated to the eigenvalue $\mu_{3}$. Then, $( \re u , \im u )$ is a basis of $\im \Pi_{\mu_{2}} \oplus \im \Pi_{\mu_{3}} = \ker ( P_{\rm com} - \mu_{2} ) \oplus \ker ( P_{\rm com} - \mu_{3} )$. In particular, $u$ cannot be a real (or purely imaginary) function. From Corollary 1.5 of \cite{BoLeMi22_01} and Theorem \ref{a20}, the solution of the evolution equation \eqref{a37} associated to $P_{\rm com}$ satisfies the following metastable behavior.

\begin{corollary}\sl \label{a29}
Consider $P_{\rm com}$ as in Theorem \ref{a20} with $h$ small enough. For all $u_{0} \in L^{2} ( \R^{2} )$, the solution $u = e^{- t P_{\rm com} / h} u_{0}$ of \eqref{a37} can be written
\begin{align}
e^{- t P_{\rm com} / h} u_{0} &= u_{1} + e^{- t \mu_{2} / h} u_{2} + e^{- t \mu_{3} / h} u_{3} + \varepsilon ( t )  \nonumber \\
&= u_{1} + e^{- t \re \mu_{2} / h} \big( \cos ( t \im \mu_{2} / h ) u_{\rm c} + \sin ( t \im \mu_{2} / h ) u_{\rm s} \big) + \varepsilon ( t ) , \label{a39}
\end{align}
with $u_{j} = \Pi_{\mu_{j}} u_{0}$ for $j= 1 , 2 , 3$, $u_{\rm c} = u_{2} + u_{3}$, $u_{\rm s} = i u_{3} - i u_{2}$ and
\begin{equation*}
\Vert \varepsilon ( t ) \Vert_{L^{2} ( \R^{2} )} \leq C e^{- t / C} \Vert u_{0} \Vert_{L^{2} ( \R^{2} )} ,
\end{equation*}
for some constant $C > 0$ independent of $t , h , u_{0}$.
\end{corollary}

If the function $u_{0}$ is real-valued, \eqref{a38} implies that $\overline{u_{2}} = u_{3}$ and then $u_{\rm c}$ and $u_{\rm s}$ are also real-valued. If in addition $u_{2}$, $u_{3}$, $u_{\rm c}$ or $u_{\rm s}$ does not vanish identically, the discussion below \eqref{a38} shows that $( u_{\rm c} , u_{\rm s} )$ is a basis of $\im \Pi_{\mu_{2}} \oplus \im \Pi_{\mu_{3}}$. In that case,
\begin{equation*}
t \longmapsto \cos ( t \im \mu_{2} / h ) u_{\rm c} + \sin ( t \im \mu_{2} / h ) u_{\rm s} ,
\end{equation*}
is a non-vanishing periodic function of period $2 \pi h \vert \im \mu_{2} \vert^{- 1}$ which reaches all the directions of $\im \Pi_{\mu_{2}} \oplus \im \Pi_{\mu_{3}}$. Then, the subprincipal term in \eqref{a39}, which measures the return to equilibrium, is oscillating. Nevertheless, this phenomenon may be difficult to see in the applications since \eqref{a85} implies that this subprincipal term decays more quickly than it oscillates.

We now construct an operator having a spectral gap with a Jordan block. For that, we consider perturbations of $P_{0}$ of the form
\begin{equation*}
P_{\rm Jor} = d_{f}^{*} \circ \big( 1 + \chi ( x , h ) \big) Id \circ d_{f} + B \qquad \text{where} \qquad B = \frac{1}{2} \big( b \cdot h \nabla + h \div \circ b \big) ,
\end{equation*}
$Id$ denotes the $2 \times 2$ identity matrix, $\chi \in C^{\infty}_{0} ( \R^{2} ; \R ) \cap S ( h^{\infty} )$ and $b \in C^{\infty}_{0} ( \R^{2} ; \R^{2} ) \cap S ( h^{\infty} )$. For $h$ small enough, such an operator falls within the general framework of \eqref{a80}.

\begin{theorem}[Jordan block]\sl \label{a70}
Let $r ( h ) = \CO ( h^{\infty} )$ be a positive function. There exist functions $\chi ( x , h ) \in C^{\infty}_{0} ( \R^{2} ; \R ) \cap S ( r )$ and $b ( x , h ) \in C^{\infty}_{0}( \R^{2} ; \R^{2} ) \cap S ( r )$ with \eqref{a31} such that, for $h$ small enough,
\begin{equation*}
\sigma ( P_{\rm Jor} ) \cap \{ z \in \C ; \ \re z < \lambda_{*} h / 2 \} = \big\{ \lambda_{1} , \lambda_{2} \big\}  \text{ of multiplicity } 1 \text{ and } 2 \text{ respectively,}
\end{equation*}
and $P_{\rm Jor}$ has a non-trivial Jordan block associated with the eigenvalue $\lambda_{2}$.
\end{theorem}

Let $\Pi_{\lambda_{1}}$ and $\Pi_{\lambda_{2}}$ be the spectral projectors of $P_{\rm Jor}$ associated to $\lambda_{1}$ and $\lambda_{2}$ respectively. From Theorem \ref{a70} and \eqref{a82}, there exists an orthonormal basis of real-valued functions, denoted $( e_{1} , e_{2} )$, of $\im \Pi_{\lambda_{2}}$ such that $\Pi_{\lambda_{2}} P_{\rm Jor} \Pi_{\lambda_{2}}$ expressed in the basis $( e_{1} , e_{2} )$ writes
\begin{equation}
\begin{pmatrix}
\lambda_{2} & \rho \\
0 & \lambda_{2}
\end{pmatrix} ,
\end{equation}
for some $\rho ( h ) \in \R \setminus \{ 0 \}$ (see \eqref{a71}). Note that $e_{1}$ and $e_{2}$ are unique modulo multiplication by $\pm 1$. By construction, the constant $\rho$ is very small. More precisely,
\begin{equation} \label{a75}
\vert \rho ( h ) \vert = \CO ( h^{\infty} \lambda_{2} ) = \CO ( h^{\infty} e^{- 2 S / h} ) .
\end{equation}
But, as for the imaginary part of the eigenvalues \eqref{a85}, this is a general fact: any Jordan block associated with a small eigenvalue of an operator of the form \eqref{a81} satisfies an estimate similar to \eqref{a75}. Indeed, all the terms in the asymptotic expansion of the interaction matrices are self-adjoint (see Section 6 of \cite{BoLeMi22_01}).

It is difficult to construct by perturbation theory an operator of the form \eqref{a77} satisfying Theorem \ref{a70}. Indeed, Lemma \ref{a22} shows that such operators enter into the setting of Theorem \ref{a20} as soon as the leading term in the perturbation theory does not vanish. This is why we consider here more general perturbations which allow to ``generate all the possible'' leading terms (see Section 4 of \cite{Sj87_01} for similar ideas in resonances theory).

Contrary to Theorem \ref{a20}, the spectral situation of Theorem \ref{a70} is unstable. Generically, a small perturbation (in the setting of \eqref{a80}) splits the double eigenvalue $\lambda_{2}$ into two non-real conjugate eigenvalues. This is  general fact concerning the Jordan blocks. Moreover, the second eigenvalue of $P_{0}$ and $P_{\rm Jor}$ is the same. The proof of Theorem \ref{a70} allows to change slightly the second eigenvalue of $P_{\rm Jor}$, but the actual statement simplifies the result.

Combining with Corollary 1.5 of \cite{BoLeMi22_01}, the time evolution equation associated to $P_{\rm Jor}$ satisfies

\begin{corollary}\sl \label{a73}
Consider $P_{\rm Jor}$ as in Theorem \ref{a70} with $h$ small enough. For all $u_{0} \in L^{2} ( \R^{2} )$, the solution $u = e^{- t P_{\rm Jor} / h} u_{0}$ of \eqref{a37} can be written
\begin{equation}
e^{- t P_{\rm Jor} / h} u_{0} = u_{1} + t e^{- t \lambda_{2} / h} u_{2} + e^{- t \lambda_{2} / h} u_{3} + \varepsilon ( t ) ,
\end{equation}
with $u_{1} = \Pi_{\lambda_{1}} u_{0}$, $u_{2} = - \rho \< e_{2} , \Pi_{\lambda_{2}} u_{0} \> e_{1}$, $u_{3} = \Pi_{\lambda_{2}} u_{0}$ and
\begin{equation*}
\Vert \varepsilon ( t ) \Vert_{L^{2} ( \R^{2} )} \leq C e^{- t / C} \Vert u_{0} \Vert_{L^{2} ( \R^{2} )} ,
\end{equation*}
for some constant $C > 0$ independent of $t , h , u_{0}$.
\end{corollary}

In particular, we have the sharp return to equilibrium result
\begin{equation} \label{a72}
\big\Vert e^{- t P_{\rm Jor} / h} - \Pi_{\lambda_{1}} \big\Vert \sim \alpha t e^{- t \lambda_{2} / h} ,
\end{equation}
in the limit $t \to + \infty$ for $h$ small enough and some positive constant $\alpha ( h ) > 0$. This estimate shows that the return to equilibrium is not purely exponentially decreasing in general and that some powers of $t$ may appear.

Until now, we have only considered the spectral gap given by exponentially small eigenvalues, corresponding to several minima. But, if we study higher eigenvalues, it is more simple to have non-real spectrum. For $\varepsilon \in \R$, consider the operator
\begin{equation*}
\CP = - h^{2} \Delta + x^{2} - 2 h + \varepsilon \big( x_{1} h \partial_{x_{2}} - x_{2} h \partial_{x_{1}} \big) .
\end{equation*}
It enters in the setting of \eqref{a80} with the Morse function $f ( x ) = x^{2} / 2$ which has a unique minimum at $x = 0$. The bottom of its spectrum is given by

\begin{proposition}\sl \label{a78}
For $h > 0$ and $\varepsilon \neq 0$, we have
\begin{equation*}
\sigma ( \CP ) \cap \{z \in \C ; \ \re z < 4 h \} = \big\{ 0 , 2 h + i \varepsilon h , 2 h - i \varepsilon h \big\} ,
\end{equation*}
and these eigenvalues are simple.
\end{proposition}

Then, this operator has a non-real spectral gap. Nevertheless, it is not given by exponentially small eigenvalues responsible of metastable dynamics. In this simple well situation, the Eyring--Kramers law only provides the asymptotic of $0$, the first eigenvalue of $\CP$. Note also that, for $\varepsilon \neq 0$ fixed, these eigenvalues do no longer satisfy \eqref{a85}.

The rest of the paper is organized as follows. In the next section, we collect some properties of the reference operator $P_{0}$ used in the sequel. Section \ref{s3} is devoted to the construction of the anti-adjoint perturbation $B$ based on properties of nodal sets. This construction allows to prove Theorem \ref{a20} (resp. Theorem \ref{a70}) in Section \ref{s4} (resp. Section \ref{s5}) combining the perturbation theory and previous results of \cite{BoLeMi22_01}. Lastly, Proposition \ref{a78} is obtained in Section \ref{s6} by direct computations.

\section{Spectral properties of $P_{0}$} \label{s2}

This part is devoted to the proof of Proposition \ref{a30} and to other technical results on $P_{0}$. From Theorem \ref{a81}, there exists $\lambda_{*} > 0$ such that, for $h$ small enough, $P_{0}$ has exactly three eigenvalues counted with multiplicity $0 = \lambda_{1} ( h ) < \lambda_{2} ( h ) \leq \lambda_{3} ( h )$ in $] - \infty , \lambda_{*} h ]$. Moreover, $\lambda_{2}$ and $\lambda_{3}$ are exponentially small. Eventually, the asymptotic
\begin{equation*}
\lambda_{2} , \lambda_{3} \sim \frac{3 \vert \mu ( s ) \vert \vert \det \hess f ( m ) \vert^{1 / 2}}{\pi \vert \det \hess f ( s ) \vert^{1 / 2}} h e^{- 2 S / h} ,
\end{equation*}
is a direct consequence of Section 7C3 of \cite{Mi19_01} (see also \cite{BoLeMi22_01,HeHiSj11_01}). We denote
\begin{equation*}
\Pi = \one_{[\lambda_{1} , \lambda_{3}]} ( P_{0} ), \qquad \Pi_{1} = \one_{\{ \lambda_{1} \}} ( P_{0} ) \qquad \text{and} \qquad \Pi_{2 3} = \one_{[\lambda_{2} , \lambda_{3}]} ( P_{0} ) ,
\end{equation*}
the spectral projectors of $P_{0}$. They satisfy $\Pi = \Pi_{1} + \Pi_{2 3}$,
\begin{equation} \label{a32}
\overline{\Pi_{\bullet} u} = \Pi_{\bullet} \overline{u} \qquad \text{and} \qquad R \Pi_{\bullet} = \Pi_{\bullet} R ,
\end{equation}
since $P_{0}$ commutes with $R$ and the complex conjugation. Here, $R$ is viewed as the rotation acting on functions (i.e. $R ( f ) = f \circ R$ for $f \in L^{2} ( \R^{d} )$).

Let $\chi \in C^{\infty}_{0} ( \R^{2}; [ 0 , 1 ] )$ be supported near $m_{1}$ with $\chi = 1$ near $m_{1}$. We set
\begin{equation} \label{a35}
\psi_{1} = \frac{\chi ( x ) e^{- f ( x ) / h}}{\Vert \chi e^{- f / h} \Vert} , \qquad \psi_{2} = \psi_{1} \circ R \qquad \text{and} \qquad \psi_{3} = \psi_{1} \circ R^{2} ,
\end{equation}
with the estimates
\begin{equation} \label{a44}
\big\Vert \chi e^{- f / h} \big\Vert \sim \sqrt{\pi h} ( \det \hess f ( m ) )^{- 1 / 4} \quad \text{and} \quad \big\Vert e^{- f / h} \big\Vert \sim \sqrt{3 \pi h} ( \det \hess f ( m ) )^{- 1 / 4} .
\end{equation}
Since $f$ is invariant by rotation, these quantities do not depend on the minimum $m$ where they are computed. The function $\psi_{j}$ is localized near $m_{j}$ from Figure \ref{f1}, and the family $( \psi_{j} )_{j}$ is orthonormal. We then set
\begin{equation*}
\varphi_{j} = \Pi \psi_{j} .
\end{equation*}
We have $\psi_{j} \in C^{\infty}_{0} ( \R^{2} )$ and $\varphi_{j} \in \CS ( \R^{2} )$. A classical result (see the proof of Proposition 2.5 of \cite{HeSj84_01}) yields
\begin{equation} \label{a11}
\varphi_{j} = \psi_{j} + \CO ( e^{- \delta / h} ) ,
\end{equation}
showing that the family $( \varphi_{j} )_{j}$ is an almost orthonormal basis of $\im \Pi$. Furthermore, \eqref{a32} implies that
\begin{equation} \label{a36}
\text{the function } \varphi_{j} \text{ is real for all } j \text{ and }\varphi_{j} = \varphi_{k} R^{j - k} \text{ for all } j , k .
\end{equation}
Eventually \eqref{a35}, \eqref{a44} and \eqref{a11} give
\begin{equation} \label{a8}
\frac{\sqrt{3}}{\Vert e^{- f / h} \Vert} e^{- f / h} = \varphi_{1} + \varphi_{2} + \varphi_{3} + \CO ( e^{- \delta / h} ) .
\end{equation}
We can now show that $\lambda_{2} = \lambda_{3}$.

\begin{lemma}\sl \label{a1}
For $h$ small enough, the second eigenvalue $\lambda_{2}$ of $P_{0}$ has multiplicity $2$.
\end{lemma}

\begin{proof}
We prove this result by contradiction. Assume that $\lambda_{2}$ has multiplicity one and let $u$ be a normalized eigenvector. From \eqref{a82}, we can always choose $u$ real-valued. In the basis $( \varphi_{j} )_{j}$ of $\im \Pi$, this function can be written
\begin{equation} \label{a6}
u = u_{1} \varphi_{1} + u_{2} \varphi_{2} + u_{3} \varphi_{3} ,
\end{equation}
for some $u_{j} \in \R$. Since $( \varphi_{j} )_{j}$ is almost orthonormal,
\begin{equation} \label{a5}
1 = \Vert u \Vert^{2} = u_{1}^{2} + u_{2}^{2} + u_{3}^{2} + \CO ( e^{- \delta / h} ) .
\end{equation}
Applying the rotation $R$, \eqref{a36} and \eqref{a6} gives
\begin{equation*}
u \circ R = u_{1} \varphi_{2} + u_{2} \varphi_{3} + u_{3} \varphi_{1} .
\end{equation*}
On the other hand, we have $P ( u \circ R ) = ( P u ) \circ R = \lambda_{2} u \circ R$. Since $\lambda_{2}$ is simple, there exists $\alpha \in \C$ such that $u \circ R = \alpha u$, that is
\begin{equation} \label{a7}
u_{1} = \alpha u_{2} , \qquad u_{2} = \alpha u_{3} \qquad \text{and} \qquad u_{3} = \alpha u_{1} .
\end{equation}
Since $u$ and $u \circ R$ are real valued, we necessarily have $\alpha \in \R$. The relation \eqref{a7} implies $u_{j} = \alpha^{3} u_{j}$ for $j = 1 , 2 , 3$. Since at least one of the $u_{j}$ is non-zero from \eqref{a5}, we get $\alpha^{3} = 1$ and then $\alpha = 1$. Thus, $u_{1} = u_{2} = u_{3}$ and $\vert u_{1} \vert = 3^{- 1 / 2} + \CO ( e^{- \delta / h } )$. On the other hand, $u$ and $e^{- f / h}$ are orthogonal since they belong to two different eigenspaces of the self-adjoint operator $P_{0}$. Combining the previous properties with \eqref{a8}, it comes
\begin{align*}
0 &= \big\vert \big\< \sqrt{3} \Vert e^{- f / h} \Vert^{- 1} e^{- f / h} , u \big\> \big\vert \\
&= \vert u_{1} \vert \big\vert \big\< \varphi_{1} + \varphi_{2} + \varphi_{3} , \varphi_{1} + \varphi_{2} + \varphi_{3} \big\> \big\vert + \CO ( e^{- \delta / h } ) \\
&= \sqrt{3} + \CO ( e^{- \delta / h} ) ,
\end{align*}
which provides a contradiction for $h$ small enough. We have just proved that $\lambda_{2}$ has multiplicity at least two. Since this multiplicity can not be larger than two, we get the lemma.
\end{proof}

\section{Construction of the anti-adjoint perturbation $B$} \label{s3}

The anti-adjoint part of $P$ is chosen of the form $B = \varepsilon \SB$ with
\begin{equation} \label{a3}
\SB = d_{f}^{*} \circ G \circ d_{f} \qquad \text{with} \qquad G = \begin{pmatrix}
0 & g \\
- g & 0
\end{pmatrix} ,
\end{equation}
with $d_{f}$ defined in \eqref{k1}, for some constant $\varepsilon ( h ) \in ] 0 , + \infty [$ and some function $g ( x , h ) \in C^{\infty}_{0} ( \R^{2} ; \R )$ fixed in the sequel.

\begin{lemma}\sl \label{a2}
The operator $\SB$ is formally anti-adjoint, $\SB e^{- f / h} = 0$ and
\begin{equation*}
\SB = \frac{1}{2} \big( b \cdot h \nabla + h \div \circ b \big) \qquad \text{with} \qquad
b ( x , h ) = \begin{pmatrix}
h \partial_{x_{2}} g - 2 g \partial_{x_{2}} f \\
- h \partial_{x_{1}} g + 2 g \partial_{x_{1}} f
\end{pmatrix} \in C^{\infty}_{0} ( \R^{2} ; \R^{2} ).
\end{equation*}
\end{lemma}

\begin{proof}
The definition \eqref{a3} of $\SB$ immediately implies that $\SB$ is formally anti-adjoint and that $\SB e^{- f / h} = 0$. Moreover, a direct computation gives
\begin{align*}
\SB ={}& \big( - h \partial_{x_{1}} + \partial_{x_{1}} f , - h \partial_{x_{2}} + \partial_{x_{2}} f \big)
\begin{pmatrix}
0 & g \\
- g & 0
\end{pmatrix} \begin{pmatrix}
h \partial_{x_{1}} + \partial_{x_{1}} f \\
h \partial_{x_{2}} + \partial_{x_{2}} f
\end{pmatrix}  \\
={}& ( - h \partial_{x_{1}} + \partial_{x_{1}} f ) g ( h \partial_{x_{2}} + \partial_{x_{2}} f ) - ( -h \partial_{x_{2}} + \partial_{x_{2}} f ) g ( h \partial_{x_{1}} + \partial_{x_{1}} f )  \\
={}& ( \partial_{x_{1}} f ) g ( h \partial_{x_{2}} ) + ( h \partial_{x_{2}} ) g ( \partial_{x_{1}} f ) - ( \partial_{x_{2}} f ) g ( h \partial_{x_{1}} ) - ( h \partial_{x_{1}} ) g ( \partial_{x_{2}} f )  \\
&- ( h \partial_{x_{1}} ) g ( h \partial_{x_{2}} ) + ( h \partial_{x_{2}} ) g ( h \partial_{x_{1}} )  \\
={}& \frac{1}{2} \big( b \cdot h \nabla + h \div b \big) ,
\end{align*}
and the lemma follows.
\end{proof}

Since we see $B$ as a perturbation of $P_{0}$  and want to use the Kato's theory, we seek the function $g \in C^{\infty}_{0} ( \R^{2} ; \R )$ such that $\Pi_{2 3} \SB \Pi_{2 3} \neq 0$. For that, let $( u , v )$ be a real-valued orthonormal basis of $\im \Pi_{2 3}$. From \eqref{a3}, we have
\begin{align}
\< \SB u , v \> &= \int_{\R^{2}} G e^{- f / h} h \nabla e^{f / h} u \cdot e^{- f / h} h \nabla e^{f / h} v \, d x  \nonumber \\
&= \int_{\R^{2}} \widetilde{g} \big( \partial_{x_{2}} \widetilde{u} \partial_{x_{1}} \widetilde{v} - \partial_{x_{1}} \widetilde{u} \partial_{x_{2}} \widetilde{v} \big) \, d x , \label{a21}
\end{align}
with
\begin{equation} \label{a15}
\widetilde{u} = e^{f / h} u , \qquad \widetilde{v} = e^{f / h} v \qquad \text{and} \qquad \widetilde{g} = h^{2} e^{- 2 f / h} g .
\end{equation}
Thus, if $\partial_{x_{2}} \widetilde{u} \partial_{x_{1}} \widetilde{v} - \partial_{x_{1}} \widetilde{u} \partial_{x_{2}} \widetilde{v}$ does not vanish identically, it is possible to find $g$ such that $\< \SB u , v \> \neq 0$. This justify the next intermediate result.

\begin{lemma}\sl \label{a16}
For $h$ small enough, we have $\partial_{x_{2}} \widetilde{u} \partial_{x_{1}} \widetilde{v} - \partial_{x_{1}} \widetilde{u} \partial_{x_{2}} \widetilde{v} \nequiv{} 0$.
\end{lemma}

\begin{proof}
We prove this lemma by contradiction. If it does not hold true, we have
\begin{equation} \label{a19}
\partial_{x_{2}} \widetilde{u} \partial_{x_{1}} \widetilde{v} - \partial_{x_{1}} \widetilde{u} \partial_{x_{2}} \widetilde{v} \equiv 0 ,
\end{equation}
for a sequence of positive $h$ which goes to $0$. Roughly speaking, this equation means that the level sets of $\widetilde{u}$ and $\widetilde{v}$ are the same. This leads to consider the nodal sets of $u$ and $v$ whose we recall now the general properties.

\begin{proposition}\sl \label{a17}
Let $w$ be a real-valued eigenvector of $P_{0}$ associated to the eigenvalue $\lambda_{2}$ (in particular, $w \in \CS ( \R^{2} ; \R )$ and $w \nequiv 0$). Then,

$1)$ the open set $\R^{2} \setminus w^{- 1} ( 0 )$ has precisely two connected components $\Omega_{\pm}^{w}$ on which $\pm w > 0$,

$2)$ the nodal set $w^{- 1} ( 0 )$ is a (unique) smooth curve without crossing on which $\nabla w \neq 0$,

$3)$ if $w_{1}$ and $w_{2}$  are two of such eigenvectors, then $w_{1}^{- 1} ( 0 ) \cap w_{2}^{- 1} ( 0 ) \neq \emptyset$.
\end{proposition}

\begin{proof}[Proof of Proposition \ref{a17}]
This result collects classical properties of nodal sets and we send the reader to the corresponding papers for the proofs. First, Section VI.6 of \cite{CoHi53_01} (see also \cite{Co23_01}) shows that $\R^{2} \setminus w^{- 1} ( 0 )$ has at most two connected components. This result, originally stated in domains, extends to our setting since the potential $V$ is confining. Moreover, if $\R^{2} \setminus w^{- 1} ( 0 )$ has only one connected component, this function has a constant sign and cannot be orthogonal to the positive function $e^{- f / h}$, an eigenvector of $P_{0}$ associated to its first eigenvalue $\lambda_{1}$. Summing up, $\R^{2} \setminus w^{- 1} ( 0 )$ has precisely two connected components.

The structure of the nodal set $w^{- 1} ( 0 )$ is described in Theorem 2.5 of \cite{Ch76_01} in the present two dimensional case (see also Theorem 2.2 in the general case). Outside of isolated critical points, $w^{- 1} ( 0 )$ is the reunion of smooth curves without crossing on which $\nabla w \neq 0$. At the critical points, a finite number of nodal curves cross and form an equiangular system. If such a critical point exists, then there will be more than two connected components in $\R^{2} \setminus w^{- 1} ( 0 )$. Thus, there is no critical point and $w^{- 1} ( 0 )$ is the reunion of smooth curves without crossing. Since such a curve is either periodical or goes to infinity, each curve in $w^{- 1} ( 0 )$ generates a connected component in $\R^{2} \setminus w^{- 1} ( 0 )$. Since this set has precisely two connected components, $w^{- 1} ( 0 )$ must be composed of an unique curve on which $\nabla w \neq 0$. Since $w$ changes sign across $w^{- 1} ( 0 )$, the connected components of $\R^{2} \setminus w^{- 1} ( 0 )$ can be labeled $\Omega_{\pm}^{w}$ in a such way that $\pm w > 0$ on $\Omega_{\pm}^{w}$. This proves $1)$ and $2)$.

It remains to show $3)$. For that, we follow the proof of Lemma 4.2 of \cite{Ch76_01}. Assume that $w_{1}^{- 1} ( 0 ) \cap w_{2}^{- 1} ( 0 ) = \emptyset$. Since $w_{1}^{- 1} ( 0 )$ is a single curve, we have $w_{1}^{- 1} ( 0 ) \subset \Omega_{-}^{w_{2}}$ or $w_{1}^{- 1} ( 0 ) \subset \Omega_{+}^{w_{2}}$. We can suppose that $w_{1}^{- 1} ( 0 ) \subset \Omega_{-}^{w_{2}}$.  Then, $\Omega_{-}^{w_{1}} \varsubsetneq \Omega_{-}^{w_{2}}$ or $\Omega_{+}^{w_{1}} \varsubsetneq \Omega_{-}^{w_{2}}$. We can suppose that $\Omega_{-}^{w_{1}} \varsubsetneq \Omega_{-}^{w_{2}}$. Eventually, by Courant's minimum principle, the first eigenvalue of the operator $P_{0}$ restricted to $\Omega_{-}^{w_{1}}$ with Dirichlet boundary condition is greater than the first eigenvalue of the operator $P_{0}$ restricted to $\Omega_{-}^{w_{2}}$ with Dirichlet boundary condition, whereas these two quantity are equal to $\lambda_{2}$. This is a contradiction and $3)$ follows.
\end{proof}

We now come back to the proof of Lemma \ref{a16}. From \eqref{a15}, the zeros of $\widetilde{u}$ (resp. $\widetilde{v}$) are those of $u$ (resp. $v$). Moreover, Lemma \ref{a17} $2)$ shows that
\begin{equation} \label{a18}
\nabla \widetilde{u} = e^{f / h} \nabla u + u \nabla e^{f / h} = e^{f / h} \nabla u \neq 0 ,
\end{equation}
on $\widetilde{u}^{- 1} ( 0 )$.  Let $x_{0}$ be a point of $\widetilde{u}^{- 1} ( 0 ) \cap \widetilde{v}^{- 1} ( 0 )$ which is not empty from Lemma \ref{a17} $3)$, and consider the curve $x ( t ) \in \R^{2}$ solution of 
\begin{equation*}
\left\{ \begin{aligned}
&\partial_{t} x ( t ) = \begin{pmatrix}
\partial_{x_{2}} \widetilde{u}  ( x ( t ) )\\
- \partial_{x_{1}} \widetilde{u} ( x ( t ) )
\end{pmatrix} ,  \\
&x( 0 ) = x_{0} .
\end{aligned} \right.
\end{equation*}
The definition of $x ( t )$ gives $\partial_{t} \widetilde{u} ( x ( t ) ) = ( \partial_{x_{1}} \widetilde{u} \partial_{x_{2}} \widetilde{u} - \partial_{x_{2}} \widetilde{u} \partial_{x_{1}} \widetilde{u} ) ( x ( t ) ) = 0$, showing that $\widetilde{u} ( x ( t ) ) = 0$ for all $t \in \R$. Combined with Lemma \ref{a17} $2)$ and \eqref{a18}, it implies that $x ( t )$ is a parametrization of $\widetilde{u}^{- 1} ( 0 )$. On the other hand, \eqref{a19} yields
\begin{equation*}
\partial_{t} \widetilde{v} ( x ( t ) ) = ( \partial_{x_{1}} \widetilde{v} \partial_{x_{2}} \widetilde{u} - \partial_{x_{2}} \widetilde{v} \partial_{x_{1}} \widetilde{u} ) ( x ( t ) ) = 0 ,
\end{equation*}
showing as before that $\widetilde{v} ( x ( t ) ) = \widetilde{v} ( x_{0} ) = 0$ for all $t \in \R$. This proves $u^{- 1} ( 0 ) = v^{- 1} ( 0 )$ from Proposition \ref{a17} $2)$. Using Proposition \ref{a17} $1)$, we deduce $\Omega_{\pm}^{u} = \Omega_{\pm}^{v}$ or $\Omega_{\pm}^{u} = \Omega_{\mp}^{v}$. It implies $\< u , v \> > 0$ or $\< u , v \> < 0$ respectively. On the other hand, $\< u , v \> = 0 $ since $( u , v )$ is orthogonal. This contradiction finishes the proof of Lemma \ref{a16}.
\end{proof}

Let $\chi \in C^{\infty}_{0} ( \R^{2} ; [ 0 , 1 ] )$ with $\supp \chi \subset B ( 0 , 1 )$ and $\chi = 1$ on $B ( 0 , 1 / 2 )$. From Lemma \ref{a16}, there exists $x_{0} = x_{0} ( h ) \in \R^{2}$ for $h$ small enough such that $( \partial_{x_{2}} \widetilde{u} \partial_{x_{1}} \widetilde{v} - \partial_{x_{1}} \widetilde{u} \partial_{x_{2}} \widetilde{v} ) ( x_{0} ) \neq 0$. By continuity ($\widetilde{u} , \widetilde{v} \in C^{\infty} ( \R^{2} )$), there exists $\nu = \nu ( h ) \in ] 0 , 1 ]$ such that $\partial_{x_{2}} \widetilde{u} \partial_{x_{1}} \widetilde{v} - \partial_{x_{1}} \widetilde{u} \partial_{x_{2}} \widetilde{v}$ does not change it sign in $B ( x_{0} , \nu )$. We then set
\begin{equation} \label{a24}
g ( x , h ) = \< x_{0} \>^{- 1} e^{- 1 / \nu} \chi \Big( \frac{x - x_{0}}{\nu} \Big) \in C^{\infty}_{0} ( \R^{2} ; \R ) ,
\end{equation}
which satisfies, for $h$ small enough and $\alpha \in \N^{2}$,
\begin{equation} \label{a25}
\forall x \in \R^{2} , \qquad \vert \partial_{x}^{\alpha} g ( x , h ) \vert = \< x_{0} \>^{- 1} \nu^{- \vert \alpha \vert} e^{- 1 / \nu} \Big\vert \chi^{( \alpha )} \Big( \frac{x - x_{0}}{\nu} \Big) \Big\vert \leq M_{\alpha} \< x_{0} \>^{- 1} ,
\end{equation}
for some constant $M_{\alpha} > 0$. Combining with Lemma \ref{a2} and $f = x^{2}$ outside a compact set, it shows that $b ( x , h )$ is a symbol of class $S ( 1 )$. Moreover, using $\widetilde{g} = h^{2} e^{- 2 f / h} g$ and \eqref{a21}, this construction yields
\begin{equation} \label{a23}
\< \SB u , v \> = \beta ,
\end{equation}
for $h$ small enough and some constant $\beta ( h ) \neq 0$.

\begin{lemma}\sl \label{a22}
In any real-valued orthonormal basis of $\im \Pi_{2 3}$, the operator $\Pi_{2 3} \SB \Pi_{2 3}$ writes
\begin{equation*}
\Pi_{2 3} \SB \Pi_{2 3} = \begin{pmatrix}
0 & \gamma \\
- \gamma & 0
\end{pmatrix} ,
\end{equation*}
for $h$ small enough and some constant $\gamma ( h ) \in \R \setminus \{ 0 \}$.
\end{lemma}

\begin{proof}
In a real-valued orthonormal basis $( e_{1} , e_{2} )$ of $\im \Pi_{2 3}$, we have
\begin{equation*}
\Pi_{2 3} \SB \Pi_{2 3} = \begin{pmatrix}
\< e_{1} , \SB e_{1} \> & \< e_{1} , \SB e_{2} \> \\
\< e_{2} , \SB e_{1} \> & \< e_{2} , \SB e_{2} \>
\end{pmatrix} .
\end{equation*}
Since $e_{1} , e_{2}$ are real-valued, \eqref{a3} gives $\< e_{1} , \SB e_{1} \> = \< e_{2} , \SB e_{2} \> = 0$ and $\< e_{2} , \SB e_{1} \> = - \< e_{1} , \SB e_{2} \>$. Let us assume that $\< e_{1} , \SB e_{2} \> = 0$ for a sequence of positive $h$ which goes to $0$. In that case, the previous relations imply $\< e_{j} , \SB e_{k} \> = 0$ for all $j , k \in \{ 1 , 2 \}$. Since $( e_{1} , e_{2} )$ is a basis of $\im \Pi_{2 3}$, it yields $\< \SB u , v \> = 0$ in contradiction with \eqref{a23}. Summing up, $\gamma ( h ) : = \< e_{1} , \SB e_{2} \> \neq 0$ for $h$ small enough.
\end{proof}

\section{Proof of Theorem \ref{a20}} \label{s4}

We now apply the perturbation theory for all $h$ fixed small enough. Let $P_{\rm com} = P_{0} + B$ with
\begin{equation}
B = \varepsilon \SB ,
\end{equation}
where $\SB$ has been constructed in Section \ref{s3}.

\begin{proposition}\sl \label{a26}
The operator $P_{\rm com}$ is closed on the domain of $P_{0}$. Moreover, for $h$ small enough, there exist $\varepsilon_{0} ( h ) > 0$ and three analytic functions $\varepsilon \longmapsto \lambda_{1} ( \varepsilon , h ) , \lambda_{2} ( \varepsilon , h ) , \lambda_{3} ( \varepsilon , h )$ defined for $\varepsilon \in [ - \varepsilon_{0} , \varepsilon_{0} ]$ with $\lambda_{1} ( \varepsilon , h ) = 0$,
\begin{equation*}
\left\{ \begin{aligned}
&\lambda_{2} ( \varepsilon , h ) = \lambda_{2} ( h ) + i \gamma ( h ) \varepsilon + \CO_{h} ( \varepsilon^{2} ) ,  \\
&\lambda_{3} ( \varepsilon , h ) = \lambda_{2} ( h ) - i \gamma ( h ) \varepsilon + \CO_{h} ( \varepsilon^{2} ) ,  \\
\end{aligned} \right.
\end{equation*}
such that
\begin{equation*}
\sigma ( P_{\rm com} ) \cap \{ z \in \C ; \ \re z < h \lambda_{*} / 2 \} = \{ \lambda_{1} , \lambda_{2} , \lambda_{3} \} ,
\end{equation*}
for $h$ small enough and $\varepsilon \in [ - \varepsilon_{0} , \varepsilon_{0} ]$.
\end{proposition}

In this statement, the notation $\CO_{h} ( 1 )$ designs a function bounded by a constant which may depend on $h$, and the eigenvalues are counted with multiplicity. The constant $\gamma ( h ) \in \R \setminus \{ 0 \}$ is whose of Lemma \ref{a22}.

\begin{proof}[Proof of Proposition \ref{a26}]
Since $B$ is a relatively compact perturbation of $P_{0}$ from Lemma \ref{a2}, the operator $P_{\rm com}$ is well-defined and closed on the domain of $P_{0}$ (see Theorem IV.1.11 of \cite{Ka76_01}). Moreover, $\varepsilon \longmapsto P_{0} + \varepsilon \SB$ is a holomorphic family of unbounded operators in the sense of Section VII of \cite{Ka76_01}. Recall that $\lambda_{1} = 0$ is a simple eigenvalue of $P_{0}$. From Lemma \ref{a1}, $\lambda_{2}$ is a double eigenvalue which is semisimple since $P_{0}$ is self-adjoint. On the other hand, Lemma \ref{a22} shows that the eigenvalues of $\Pi_{2 3} \SB \Pi_{2 3}$ are $\pm i \gamma$. Since $\gamma \neq 0$, these eigenvalues are different. Then, by the perturbation theory of spectrum, more precisely the perturbation theory of finite systems of eigenvalues (see Section VII.1.3 of \cite{Ka76_01}) and the reduction process for semisimple eigenvalues (see Section II.2.3 of \cite{Ka76_01}), there exist analytic functions $\varepsilon \longmapsto \lambda_{1} ( \varepsilon , h ) , \lambda_{2} ( \varepsilon , h ) , \lambda_{3} ( \varepsilon , h )$ defined for $\varepsilon \in [ - \varepsilon_{0} , \varepsilon_{0} ]$ with $\varepsilon_{0} ( h ) > 0$ such that
\begin{equation*}
\left\{ \begin{aligned}
&\lambda_{1} ( 0 , h ) = 0 , \\
&\lambda_{2} ( \varepsilon , h ) = \lambda_{2} ( h ) + i \gamma ( h ) \varepsilon + \CO_{h} ( \varepsilon^{2} ) ,  \\
&\lambda_{3} ( \varepsilon , h ) = \lambda_{2} ( h ) - i \gamma ( h ) \varepsilon + \CO_{h} ( \varepsilon^{2} ) ,  \\
\end{aligned} \right.
\end{equation*}
and $\sigma ( P_{\rm com} ) \cap \{ z \in \C ; \ \re z < h \lambda_{*} / 2 \} = \{ \lambda_{1} , \lambda_{2} , \lambda_{3} \}$. Since $0$ is always an eigenvalue of $P_{\rm com}$ by Lemma \ref{a2}, we have $\lambda_{1} ( \varepsilon , h ) = 0$ for all $\varepsilon \in [ - \varepsilon_{0} , \varepsilon_{0} ]$ after a possible shrinking of $\varepsilon_{0}$.
\end{proof}

The asymptotic expansions in Proposition \ref{a26} and $\gamma ( h ) \in \R \setminus \{ 0 \}$ yield that
\begin{equation} \label{a27}
\im \lambda_{2} ( \varepsilon , h ) \neq 0 \qquad \text{and} \qquad \im \lambda_{3} ( \varepsilon , h ) \neq 0 ,
\end{equation}
for all $\varepsilon \in [ - \varepsilon_{1} , \varepsilon_{1} ] \setminus \{ 0 \}$ with $\varepsilon_{1} ( h ) > 0$ small enough. We eventually choose
\begin{equation} \label{a28}
\varepsilon ( h ) = \min \big( \varepsilon_{0} ( h ) , \varepsilon_{1} ( h ) , r ( h ) \big) .
\end{equation}
Thus, $b ( x , h )$ is a symbol of order at most $r ( h )$ from Lemma \ref{a2} and \eqref{a25}. In the domain $\{ z \in \C ; \ \re z < \lambda_{*} h / 2 \}$, $P_{\rm com}$ has three eigenvalues $\mu_{1} ( h ) = 0$, $\mu_{2} ( h ) = \lambda_{2} ( \varepsilon ( h ) , h )$ and $\mu_{3} ( h ) = \lambda_{3} ( \varepsilon ( h ) , h )$ with $\im \mu_{2} \neq 0$ and $\im \mu_{3} \neq 0$. From \eqref{a82}, we automatically have $\mu_{3} = \overline{\mu_{2}}$. Finally, we can write
\begin{equation*}
P_{\rm com} - z = ( B ( P_{0} - z )^{- 1} - 1 ) ( P_{0} - z ) ,
\end{equation*}
for $z \in B ( 0 , 1 ) \setminus \sigma ( P_{0} )$ with
\begin{equation*}
B ( P_{0} - z )^{- 1} = B ( P_{0} + i )^{- 1} \big( 1 + ( z + i ) ( P_{0} - z )^{- 1} \big) = \CO \big( r \dist ( z , \sigma ( P_{0} ) )^{- 1} \big) ,
\end{equation*}
from Lemma \ref{a2}. By a classical argument, it implies that $\mu_{2} = \lambda_{2} + \CO ( r )$ and finishes the proof of Theorem \ref{a20}.

\section{Proof of Theorem \ref{a70}} \label{s5}

To find a setting with a Jordan block, we consider operators of the form 
\begin{equation} \label{a74}
P_{\tau} = P_{0} + \tau_{1} P_{1} + \tau_{2} P_{2} + \tau_{3} P_{3} + \tau_{4} \SB ,
\end{equation}
where $\tau = ( \tau_{1} , \tau_{2} , \tau_{3} , \tau_{4} ) \in \R^{4}$, $\SB$ has been constructed in Section \ref{s3} and $P_{\nu}$ is as follows. For $\nu = 1 , 2 , 3$, let $\chi_{\nu} \in C^{\infty}_{0} ( \R^{2} ; \R )$ be supported near $s_{\nu}$ and equal to $1$ in a neighborhood of $s_{\nu}$. We also assume that $\chi_{2} = \chi_{1} \circ R$ and $\chi_{3} = \chi_{1} \circ R^{2}$ (see Figure \ref{f1}). Then, $P_{\nu}$ is defined by
\begin{equation} \label{a48}
P_{\nu} = d_{f}^{*} \circ \chi_{\nu} \circ d_{f} .
\end{equation}
Summing up, $P_{\tau}$ can be written
\begin{equation*}
P_{\tau} = d_{f}^{*} \circ \chi \circ d_{f} + \tau_{4} \SB ,
\end{equation*}
with $\chi = 1 + \tau_{1} \chi_{1} + \tau_{2} \chi_{2} + \tau_{3} \chi_{3}$. Thus, for $\tau$ small enough, $P_{\tau}$ enters into the setting of \eqref{a80}.

For $j = 1 , 2 , 3$, let $\phi_{j} ( x )$ denote the global quasimode of $P_{0}$ supported near the connected component of $\{ f < f ( s ) \}$ containing $m_{j}$ constructed in Section 4 of \cite{BoLeMi22_01} (see also \cite{LePMi20}). More precisely, this real-valued function can be written
\begin{equation} \label{a40}
\phi_{j} ( x ) = \theta_{j} ( x ) \big( v_{j} ( x ) + 1 \big) e^{- f ( x ) / h} = \widetilde{v}_{j} ( x ) e^{- f ( x ) / h} ,
\end{equation}
where $\theta_{j} \in C^{\infty}_{0} ( \R^{2} )$ is a plateau function near the connected component of $\{ f < f ( s ) \}$ containing $m_{j}$ and $v_{j} \in C^{\infty} ( \R^{2} )$ is given near the support of $\theta_{j}$ by
\begin{equation} \label{a41}
v_{j} ( x ) = \left\{ \begin{aligned}
&C_{0}^{- 1} \int_{0}^{\ell_{s}^{j} ( x , h )} \zeta ( r ) e^{- r^{2} / 2 h} d r \qquad  &&\begin{aligned}
&\text{near } s, \text{ one of the two} \\
&\text{saddle points close to } m_{j}, \end{aligned}  \\
&1 &&\begin{aligned}\text{outside,}\end{aligned}
\end{aligned} \right.
\end{equation}
Here, we say that ``a saddle point $s$ is close to a minimum $m$'' iff $s$ is in the closure of the connected component of $\{ f < f ( s ) \}$ containing $m$ (see Figure \ref{f1}). The function $\zeta \in C^{\infty}_{0} ( \R ; [ 0 , 1 ] )$ is even and satisfies $\zeta ( r ) = 1$ for $r$ near $0$,
\begin{equation*}
C_{0}  = \int_{0}^{+ \infty} \zeta ( r ) e^{- r^{2} / 2 h} d r = \sqrt{\frac{\pi h}{2}} \big( 1 + \CO ( e^{- \delta / h } ) \big) .
\end{equation*}
The function $\ell_{s}^{j} ( x , h ) \simeq \ell_{s , 0}^{j} ( x ) + \ell_{s , 1}^{j} ( x ) h + \cdots$ is smooth with $\ell_{s , 0}^{j} ( s ) = 0$ and $\nabla \ell_{s , 0}^{j} ( s ) \neq 0$. As in \eqref{a35}, we can make these constructions so that
\begin{equation} \label{a61}
\phi_{2} = \phi_{1} \circ R \qquad \text{and} \qquad \phi_{3} = \phi_{1} \circ R^{2} .
\end{equation}
We choose $\chi_{\nu}$ in \eqref{a48} such that $\chi_{\nu} = 1$ near the support of $\theta_{j} \nabla v_{j}$ if $s_{\nu}$ is close to $m_{j}$. By comparison with Section \ref{s2}, we have $\phi_{j} = 2 \psi_{j} + \CO ( e^{- \delta / h} )$ for some $\delta > 0$, but $\phi_{j}$ is a better quasimode than $\psi_{j}$ (see Lemma \ref{a42} below).

We define the geometric quantities $S = f ( s ) - f ( m )$ and
\begin{equation*}
C_{1} = \frac{2 \vert \mu ( s ) \vert}{\vert \det \hess f ( s ) \vert^{1 / 2}} ,
\end{equation*}
where $\mu ( s )$ is given above Proposition \ref{a30}. The quasimodes $\phi_{j}$'s satisfy

\begin{lemma}\sl \label{a42}
For all $\nu , j , k \in \{ 1 , 2 , 3 \}$, we have
\begin{gather}
\< P_{\nu} \phi_{j} , \phi_{k} \> \sim \left\{ \begin{aligned}
&C_{1} h^{2} e^{- 2 S / h} &&\text{if } s_{\nu} \text{ is close to } m_{j} = m_{k} ,  \\
&- C_{1} h^{2} e^{- 2 S / h} \qquad &&\text{if } s_{\nu} \text{ is close to } m_{j} \neq m_{k} ,  \\
&0 &&\text{otherwise,}
\end{aligned} \right.   \label{a45} \\
\Vert P_{0} \phi_{j} \Vert^{2} = \CO ( h^{\infty} ) e^{- 2 S / h} \qquad \text{and} \qquad \Vert P_{\nu} \phi_{j} \Vert^{2} = \CO ( h^{\infty} ) e^{- 2 S / h} .  \label{a46}
\end{gather}
\end{lemma}

Here, the notation ``$s_{\nu}$ is close to $m_{j} \neq m_{k}$'' means that $j \neq k$ and that $s_{\nu}$ is close to $m_{j}$ and $m_{k}$. Roughly speaking, it means that $s_{\nu}$ is between $m_{j}$ and $m_{k}$ (see Figure \ref{f1}).

\begin{proof}
This result is similar to Proposition 5.1 of \cite{BoLeMi22_01} (see also Section 4B of \cite{LePMi20}). We only explain here the ideas of the proof and the necessary changes, and we send the reader to \cite{BoLeMi22_01} for the details.

Combining \eqref{a48} and \eqref{a40} leads to
\begin{equation} \label{a52}
\big\< P_{\nu} \phi_{j} , \phi_{k} \big\> = \big\< \chi_{\nu} d_{f} \phi_{j} , d_{f} \phi_{k} \big\> = \big\< \chi_{\nu} e^{- f / h} h \nabla \widetilde{v}_{j} , e^{- f / h} h \nabla \widetilde{v}_{k} \big\> .
\end{equation}
Using $\nabla \widetilde{v}_{j} = ( v_{j} + 1 ) \nabla \theta_{j} + \theta_{j} \nabla v_{j}$ and $e^{- f / h} = \CO ( e^{- ( S + \delta ) / h} )$ on the support of $( v_{j} + 1 ) \nabla \theta_{j}$, the previous equation becomes
\begin{equation} \label{a50}
\big\< P_{\nu} \phi_{j} , \phi_{k} \big\> = h^{2} \int \chi_{\nu} \theta_{j} \theta_{k} \nabla v_{j} \cdot \nabla v_{k} e^{- 2 f / h} d x + \CO \big( e^{- 2 ( S + \delta ) / h} \big) .
\end{equation}
From \eqref{a41}, we have on the support of $\theta_{j}$
\begin{equation*}
\nabla v_{j} = \sum_{s \text{ close to } m_{j}} C_{0}^{- 1} \zeta ( \ell_{s}^{j} ) e^{- ( \ell_{s}^{j} )^{2} / 2 h} \nabla \ell_{s}^{j} .
\end{equation*}
If $s_{\nu}$ is close to $m_{j} = m_{k}$, \eqref{a50} writes
\begin{equation*}
\big\< P_{\nu} \phi_{j} , \phi_{k} \big\> = h^{2} C_{0}^{- 2} \int \theta_{j}^{2} \zeta ( \ell_{s_{\nu}}^{j} )^{2} \vert \nabla \ell_{s_{\nu}}^{j} \vert^{2} e^{- 2 \big( f + \frac{( \ell_{s_{\nu}}^{j} )^{2}}{2} \big) / h} d x + \CO \big( e^{- 2 ( S + \delta ) / h} \big) .
\end{equation*}
The asymptotic of such an integral has been obtained in Equation (5.4) of \cite{BoLeMi22_01} using the Laplace method. This computation gives
\begin{equation} \label{a51}
\big\< P_{\nu} \phi_{j} , \phi_{k} \big\> \sim C_{1} h^{2} e^{- 2 S / h} ,
\end{equation}
when $s_{\nu}$ is close to $m_{j} = m_{k}$. Assume now that $s_{\nu}$ is close to $m_{j}$ and $m_{k}$ with $m_{j} \neq m_{k}$. In that case, we have $\ell_{s_{\nu}}^{j} = - \ell_{s_{\nu}}^{k}$ (see the discussion below (4.6) of \cite{BoLeMi22_01}). Then, \eqref{a50} and the parity of $\zeta$ give
\begin{equation*}
\big\< P_{\nu} \phi_{j} , \phi_{k} \big\> = - h^{2} C_{0}^{- 2} \int \theta_{j} \theta_{k} \zeta ( \ell_{s_{\nu}}^{j} )^{2} \vert \nabla \ell_{s_{\nu}}^{j} \vert^{2} e^{- 2 \big( f + \frac{( \ell_{s_{\nu}}^{j} )^{2}}{2} \big) / h} d x + \CO \big( e^{- 2 ( S + \delta ) / h} \big) .
\end{equation*}
As before, the Laplace method implies
\begin{equation} \label{a54}
\big\< P_{\nu} \phi_{j} , \phi_{k} \big\> \sim - C_{1} h^{2} e^{- 2 S / h} ,
\end{equation}
when $s_{\nu}$ is close to $m_{j} \neq m_{k}$. Finally, if $s_{\nu}$ is not close to $m_{j}$ or $m_{k}$, we directly get from \eqref{a52} and the support properties of $\chi_{\nu}$, $\theta_{j}$ and $\theta_{k}$ that
\begin{equation} \label{a53}
\big\< P_{\nu} \phi_{j} , \phi_{k} \big\> = 0 ,
\end{equation}
in that case. Summing up, \eqref{a45} follows from \eqref{a51}, \eqref{a54} and \eqref{a53}.

It remains to show \eqref{a46}. The first estimate is a direct consequence of Proposition 5.1 $ii)$ and $iii)$ of \cite{BoLeMi22_01}. On the other hand, using \eqref{a40} and $P_{\nu} e^{- f / h } = 0$, we deduce
\begin{equation*}
P_{\nu} \phi_{j} = [ P_{\nu} , \theta_{j} ] ( v_{j} + 1 ) e^{- f / h} + \theta_{j} P_{\nu} \big( v_{j} e^{- f / h} \big) .
\end{equation*}
Since $e^{- f / h} = \CO ( e^{- ( S + \delta ) / h} )$ on the support of $(v_{j} + 1 ) \nabla \theta_{j} $, the first term is $\CO ( e^{- ( S + \delta ) / h} )$ in $L^{2}$ norm. Concerning the second term, we remark that $\chi_{\nu}$ is constant (equal to $0$ or $1$) near each connected component of the support of $\theta_{j} \nabla v_{j}$ if the support of $\theta_{j}$ has been chosen sufficiently close to the connected component of $\{ f < f ( s ) \}$ containing $m_{j}$. Then, \eqref{a48} and \eqref{a41} lead to
\begin{align*}
\theta_{j} P_{\nu} \big( v_{j} e^{- f / h} \big) &= \theta_{j} d_{f}^{*} \chi_{\nu} d_{f} v_{j} e^{- f / h} = \theta_{j} d_{f}^{*} \chi_{\nu} e^{- f / h} h \nabla v_{j} \\
&= \chi_{\nu} \theta_{j} d_{f}^{*} e^{- f / h} h \nabla v_{j} = \chi_{\nu} \theta_{j} P_{0} \big( v_{j} e^{- f / h} \big) .
\end{align*}
It is proved below (5.5) of \cite{BoLeMi22_01} that $\theta_{j} P_{0} ( v_{j} e^{- f / h} ) = \CO ( h^{\infty} ) e^{- S / h}$. Summing up, 
\begin{equation*}
P_{\nu} \phi_{j} = \CO ( h^{\infty} ) e^{- S / h} ,
\end{equation*}
and \eqref{a46} follows.
\end{proof}

We construct a basis of the $2$-dimensional spectral space of $P_{\tau}$ associated to the eigenvalues close to $\lambda_{2}$. We set
\begin{equation} \label{a47}
\left\{ \begin{aligned}
\widetilde{e}_{1} ( x ) &= \frac{1}{\sqrt{8} \Vert e^{- f / h} \Vert} \big( 2 \phi_{1} ( x ) - \phi_{2} ( x ) - \phi_{3} ( x ) \big) , \\
\widetilde{e}_{2} ( x ) &= \frac{\sqrt{3}}{\sqrt{8} \Vert e^{- f / h} \Vert} \big( \phi_{2} ( x ) - \phi_{3} ( x ) \big) ,
\end{aligned} \right.
\end{equation}
with $\Vert e^{- f / h} \Vert$ estimated in \eqref{a44}. The idea behind this choice of functions is that
\begin{equation} \label{a63}
\frac{1}{\sqrt{3}} \begin{pmatrix}
1 \\
1 \\
1
\end{pmatrix} , \quad
\frac{1}{\sqrt{6}} \begin{pmatrix}
2 \\
- 1 \\
- 1
\end{pmatrix} , \quad \frac{1}{\sqrt{2}} \begin{pmatrix}
0 \\
1 \\
- 1
\end{pmatrix} ,
\end{equation}
form an orthonormal basis of $\R^{3}$ (the first vector corresponding to $\Vert e^{- f / h} \Vert^{- 1} e^{- f / h}$). For $\nu \in \{ 1 , 2 , 3 \}$, let $\widetilde{\CP}_{\nu} \in M_{2 \times 2} ( \R )$ be the matrix of coefficients
\begin{equation*}
( \widetilde{\CP}_{\nu} )_{j , k} = C_{2}^{- 1} h^{- 1} e^{2 S / h} \< P_{\nu} \widetilde{e}_{j} , \widetilde{e}_{k} \> \qquad \text{with} \qquad C_{2} = \frac{\vert \mu ( s ) \vert \vert \det \hess f ( m ) \vert^{1 / 2}}{4 \pi \vert \det \hess f ( s ) \vert^{1 / 2}} .
\end{equation*}
The asymptotic of these matrices are provided by the next result.

\begin{lemma}\sl \label{a55}
For all $j , k \in \{ 1 , 2 \}$, we have
\begin{equation} \label{a43}
\< \widetilde{e}_{j} , \widetilde{e}_{k} \> = \delta_{j , k} + \CO ( e^{- \delta / h } ) .
\end{equation}
Moreover, the matrices $\widetilde{\CP}_{\nu}$ satisfy modulo $o ( 1 )$ terms
\begin{equation} \label{a56}
\widetilde{\CP}_{1} = \begin{pmatrix}
3 & - \sqrt{3} \\
- \sqrt{3} & 1 \end{pmatrix} , \qquad
\widetilde{\CP}_{2} = \begin{pmatrix}
0 & 0 \\
0 & 4 \end{pmatrix} \qquad \text{and} \qquad
\widetilde{\CP}_{3} = \begin{pmatrix}
3 & \sqrt{3} \\
\sqrt{3} & 1 \end{pmatrix} .
\end{equation}
\end{lemma}

\begin{proof}
From \eqref{a40}, we have $\phi_{\nu} = 2 e^{- f / h}$ near $m_{\nu}$ and $\phi_{\nu} = \CO ( e^{- \delta / h } )$ outside. It implies
\begin{equation*}
\Vert e^{- f / h} \Vert^{- 2} \< \phi_{j} , \phi_{k} \> = \frac{4}{3} \delta_{j , k} + \CO ( e^{- \delta / h } ) ,
\end{equation*}
thanks to \eqref{a44}. Combining this relation with \eqref{a47}, we get \eqref{a43}.

To show \eqref{a56}, it is enough to combine \eqref{a45} and \eqref{a47}. For instance
\begin{align*}
\< P_{3} \widetilde{e}_{1} , \widetilde{e}_{2} \> &= \frac{\sqrt{3}}{8 \Vert e^{- f / h} \Vert^{2}} \big\< P_{3} ( 2 \phi_{1} - \phi_{2} - \phi_{3} ) , ( \phi_{2} - \phi_{3} ) \big\>  \\
&= \frac{\sqrt{3}}{8 \Vert e^{- f / h} \Vert^{2}} \big( 2 \< P_{3} \phi_{1} , \phi_{2} \> - 2 \< P_{3} \phi_{1} , \phi_{3} \> - \< P_{3} \phi_{2} , \phi_{2} \>  \\
&\qquad \qquad \qquad \qquad \qquad \quad + \< P_{3} \phi_{2} , \phi_{3} \> - \< P_{3} \phi_{3} , \phi_{2} \> + \< P_{3} \phi_{3} , \phi_{3} \> \big)    \\
&= \frac{\sqrt{3} C_{1} h^{2} e^{- 2 S / h}}{8 \Vert e^{- f / h} \Vert^{2}} ( 0 + 2 + 0 + 0 + 0 + 1 ) + o \big( h e^{- 2 S / h} \big)  \\
&= \sqrt{3} C_{2} h e^{- 2 S / h} + o \big( h e^{- 2 S / h} \big) ,
\end{align*}
thanks to \eqref{a44} and
\begin{equation*}
C_{2} = \frac{C_{1} \vert \det \hess f ( m ) \vert^{1 / 2}}{8 \pi} .
\end{equation*}
This provides the desired asymptotic of $( \widetilde{\CP}_{3} )_{1 , 2}$. The other coefficients can be computed the same way.
\end{proof}

We now apply the perturbation theory for $h$ fixed small enough. For $\tau$ small enough (depending on $h$), $P_{\tau}$ has $3$ small eigenvalues counted with multiplicity: $\lambda_{1}^{\tau} = 0$ associated to the eigenspace $e^{f / h} \C$ and $\lambda_{2} ^{\tau} , \lambda_{3}^{\tau}$ such that $\lambda_{j}^{\tau} \to \lambda_{2}$ as $\tau \to 0$ for $j = 2 , 3$. Mimicking Section \ref{s2}, we introduce the spectral projectors
\begin{equation*}
\Pi^{\tau} = \one_{\{ \lambda_{1} , \lambda_{2} , \lambda_{3} \}} ( P_{\tau} ), \qquad \Pi_{1}^{\tau} = \one_{\{ \lambda_{1} \}} ( P_{\tau} ) \qquad \text{and} \qquad \Pi_{2 3}^{\tau} = \one_{\{ \lambda_{2} , \lambda_{3} \}} ( P_{\tau} ) ,
\end{equation*}
which satisfy $\Pi^{\tau} = \Pi_{1}^{\tau} + \Pi_{2 3}^{\tau}$, $\Pi_{1}^{\tau} = \Pi_{1}$ and $\overline{\Pi_{\bullet}^{\tau} u} = \Pi_{\bullet}^{\tau} \overline{u}$ as in \eqref{a32}. Moreover, $\tau \to \Pi_{\bullet}^{\tau}$ is analytic in a real neighborhood of $0$ which may depend on $h$. We define
\begin{equation*}
\widehat{e}_{1}^{\tau} = \Pi_{2 3}^{\tau} \widetilde{e}_{1} \qquad \text{and} \qquad \widehat{e}_{2}^{\tau} = \Pi_{2 3}^{\tau} \widetilde{e}_{2} .
\end{equation*}
for $j = 1 , 2$. These real-valued functions satisfy

\begin{lemma}\sl \label{a60}
For $j = 1 , 2$ and $\tau$ small enough (depending on $h$), we have
\begin{equation*}
\widehat{e}_{j}^{\tau} = \widetilde{e}_{j} + \CO ( h^{\infty} ) e^{- S / h} ,
\end{equation*}
in $H^{2} ( \R^{2} )$.
\end{lemma}

\begin{proof}
Using the Cauchy formula, we can write
\begin{align*}
\widehat{e}_{j}^{\tau} &= \Pi_{2 3} \widetilde{e}_{j} - \frac{1}{2 i \pi} \oint_{\gamma} ( P_{\tau} - z )^{- 1} \, d z \,  \widetilde{e}_{j} + \frac{1}{2 i \pi} \oint_{\gamma} ( P_{0} - z )^{- 1} \, d z \,  \widetilde{e}_{j}  \\
&= \Pi_{2 3} \widetilde{e}_{j} + \frac{1}{2 i \pi} \oint_{\gamma} ( P_{\tau} - z )^{- 1} ( P_{\tau} - P_{0} ) ( P_{0} - z )^{- 1} \, d z \,  \widetilde{e}_{j} ,
\end{align*}
where $\gamma$ is a simple loop around $\lambda_{2}$, oriented counterclockwise, which depends on $h$ but not on $\tau$. This implies
\begin{equation} \label{a62}
\widehat{e}_{j}^{\tau} = \Pi_{2 3} \widetilde{e}_{j} + \CO_{h} ( \tau ) = \Pi_{2 3} \widetilde{e}_{j} + \CO ( h^{\infty} ) e^{- S / h} ,
\end{equation}
in $H^{2} ( \R^{2} )$ for $\tau$ small enough.

Since $\Pi_{2 3} = \Pi - \Pi_{1}$, \eqref{a62} gives
\begin{equation*}
\widehat{e}_{j}^{\tau} = \Pi \widetilde{e}_{j} - \Pi_{1} \widetilde{e}_{j} + \CO ( h^{\infty} ) e^{- S / h} .
\end{equation*}
Using \eqref{a61} and that $f$ is invariant under $R$, we get
\begin{align*}
\Pi_{1} \widetilde{e}_{1} &= e^{- f / h } \frac{1}{\sqrt{8} \Vert e^{- f / h} \Vert^{3}} \big\< e^{- f / h} , 2 \phi_{1} - \phi_{2} - \phi_{3} \big\>  \\
&= e^{- f / h } \frac{1}{\sqrt{8} \Vert e^{- f / h} \Vert^{3}} \big\< e^{- f / h} , 2 \phi_{1} - \phi_{1} R - \phi_{1} R^{2} \big\>  \\
&= e^{- f / h } \frac{1}{\sqrt{8} \Vert e^{- f / h} \Vert^{3}} \big\< e^{- f / h} , 2 \phi_{1} - \phi_{1} - \phi_{1} \big\>  \\
&= 0 .
\end{align*}
The same way, $\Pi_{1} \widetilde{e}_{2} = 0$. These relations follow in fact from \eqref{a63}. Summing up,
\begin{equation} \label{a64}
\widehat{e}_{j}^{\tau} = \Pi \widetilde{e}_{j} + \CO ( h^{\infty} ) e^{- S / h} ,
\end{equation}
in $H^{2} ( \R^{2} )$ for $\tau$ small enough.

Finally, we work as in \cite[Lemma 4.9]{LePMi20} to remove the projector $\Pi$ (see also \cite[(5.12)]{BoLeMi22_01}). The Cauchy formula gives
\begin{align*}
\Pi \widetilde{e}_{j} &= \widetilde{e}_{j} - \frac{1}{2 i \pi} \oint_{\partial B ( 0 , \lambda_{*} h / 2 )} ( P_{0} - z )^{- 1} \, d z \,  \widetilde{e}_{j} - \frac{1}{2 i \pi} \oint_{\partial B ( 0 , \lambda_{*} h / 2 )} z^{- 1} \, d z \,  \widetilde{e}_{j}  \\
&= \widetilde{e}_{j} - \frac{1}{2 i \pi} \oint_{\partial B ( 0 , \lambda_{*} h / 2 )} z ( P_{0} - z )^{- 1} \, d z \, P_{0} \widetilde{e}_{j} .
\end{align*}
Combining with \eqref{a46}, \eqref{a47} and \eqref{a64}, the lemma follows.
\end{proof}

From \eqref{a47}, \eqref{a63} and Lemma \ref{a60}, $( \widehat{e}_{1}^{\tau} , \widehat{e}_{2}^{\tau} )$ is almost orthonormal and then is a basis of $\im \Pi_{2 3}^{\tau}$. We orthonormalize $( \widehat{e}_{1}^{\tau} , \widehat{e}_{2}^{\tau} )$ into $( e_{1}^{\tau} , e_{2}^{\tau} )$ by the Gram--Schmidt process. It means
\begin{equation} \label{a65}
e_{1}^{\tau} = \Vert \widehat{e}_{1}^{\tau} \Vert^{- 1} \widehat{e}_{1}^{\tau} \qquad \text{and} \qquad e_{2}^{\tau} = \big\Vert \widehat{e}_{2}^{\tau} - \< e_{1}^{\tau} , \widehat{e}_{2}^{\tau} \> e_{1}^{\tau} \big\Vert^{- 1} \big( \widehat{e}_{2}^{\tau} - \< e_{1}^{\tau} , \widehat{e}_{2}^{\tau} \> e_{1}^{\tau} \big) .
\end{equation}
In particular, $( e_{1}^{\tau} , e_{2}^{\tau} )$ is a orthonormal basis of $\im \Pi_{2 3}^{\tau}$, $e_{j}^{\tau}$ is real-valued and $\tau \to e_{j}^{\tau}$ is analytic for $j = 1 , 2$ and $\tau$ near $0$. We now define the interaction matrix $Q ( \tau )$. More precisely, 
\begin{equation*}
\text{let } Q ( \tau ) \text{ be the matrix of the operator } \Pi_{2 3}^{\tau} P_{\tau} \Pi_{2 3}^{\tau} \text{ expressed in the basis } ( e_{1}^{\tau} , e_{2}^{\tau} ) .
\end{equation*}
More prosaically, it means that $Q_{j , k} ( \tau ) = \< e_{j}^{\tau} , P_{\tau} e_{k}^{\tau} \>$. From the previous discussion, $Q ( \tau )$ is well-defined for $\tau$ small,
\begin{equation} \label{a57}
Q \in C^{\infty} ( \R^{4} ; M_{2 \times 2}( \R ) ) ,
\end{equation}
and $Q ( 0 ) = \lambda_{2} Id$. Moreover, its partial derivates satisfy

\begin{lemma}\sl \label{a58}
We have
\begin{equation} \label{a68}
\partial_{\tau_{\nu}} Q ( 0 ) = C_{2} h e^{- 2 S / h} \big( \widetilde{\CP}_{\nu} + o ( 1 ) \big) ,
\end{equation}
for $\nu = 1 , 2 , 3$ and
\begin{equation} \label{a69}
\partial_{\tau_{4}} Q ( 0 ) = \begin{pmatrix}
0 & \gamma \\
- \gamma & 0
\end{pmatrix} ,
\end{equation}
where the $\widetilde{\CP}_{\nu}$'s are defined in Lemma \ref{a55} and $\gamma ( h ) \in \R \setminus \{ 0 \}$ is the constant given by Lemma \ref{a22} and associated to the basis $( e_{1}^{0} , e_{2}^{0} )$ of $\im \Pi_{2 3}$.
\end{lemma}

\begin{proof}
We compute these derivates using the classical trick in the reduction process (see Section II.2.3 of \cite{Ka76_01}). We can write
\begin{align}
\partial_{\tau_{\nu}} Q_{j ,k} ( 0 ) &= \partial_{\tau_{\nu}} \big( Q_{j ,k} - \delta_{j , k} \lambda_{2} \big) ( 0 ) = \partial_{\tau_{\nu}} \big\< e_{j}^{\tau} , ( P_{\tau} - \lambda_{2} ) e_{k}^{\tau} \big\> ( 0 )    \nonumber \\
&= \big\< ( \partial_{\tau_{\nu}} e_{j}^{0} ) , ( P_{0} - \lambda_{2} ) e_{k}^{0} \big\> + \big\< e_{j}^{0} , ( \partial_{\tau_{\nu}} P_{\tau} ) ( 0 ) e_{k}^{0} \big\> + \big\< e_{j}^{0} , ( P_{0} - \lambda_{2} ) ( \partial_{\tau_{\nu}} e_{k}^{0} ) \big\>  \nonumber \\
&= \big\< e_{j}^{0} , ( \partial_{\tau_{\nu}} P_{\tau} ) ( 0 ) e_{k}^{0} \big\> ,  \label{a59}
\end{align}
since $( P_{0} - \lambda_{2} ) e_{j}^{0} = ( P_{0} - \lambda_{2} ) e_{k}^{0} = 0$. For $\nu = 1 , 2 , 3$, \eqref{a59} gives
\begin{equation} \label{a67}
\partial_{\tau_{\nu}} Q_{j ,k} ( 0 ) = \big\< e_{j}^{0} , P_{\nu} e_{k}^{0} \big\> ,
\end{equation}
with $P_{\nu}$ defined in \eqref{a48}.

From \eqref{a43} and Lemma \ref{a60}, we have
\begin{equation*}
\Vert \widehat{e}_{1}^{\tau} \Vert = 1 + \CO ( e^{- \delta / h} ) \qquad \text{and} \qquad \big\Vert \widehat{e}_{2}^{\tau} - \< e_{1}^{\tau} , \widehat{e}_{2}^{\tau} \> e_{1}^{\tau} \big\Vert = 1 + \CO ( e^{- \delta / h} ) .
\end{equation*}
Using again Lemma \ref{a60}, \eqref{a65} becomes
\begin{equation} \label{a66}
\left\{ \begin{aligned}
e_{1}^{\tau} &= \widetilde{e}_{1} + \CO ( e^{- \delta / h} ) \widetilde{e}_{1} + \CO ( h^{\infty} ) e^{- S / h} ,  \\
e_{2}^{\tau} &= \widetilde{e}_{2} + \CO ( e^{- \delta / h} ) \widetilde{e}_{1} + \CO ( e^{- \delta / h} ) \widetilde{e}_{2} + \CO ( h^{\infty} ) e^{- S / h} ,
\end{aligned} \right.
\end{equation}
where the $\CO ( e^{- \delta / h} )$'s are constants. Then, \eqref{a68} follows from \eqref{a56}, \eqref{a67} and \eqref{a66}.

On the other hand, \eqref{a67} gives $\partial_{\tau_{4}} Q_{j ,k} ( 0 ) = \big\< e_{j}^{0} , \SB e_{k}^{0} \big\>$. In other words, $\partial_{\tau_{4}} Q ( 0 )$ is the operator $\Pi_{2 3} \SB \Pi_{2 3}$ expressed in the basis $( e_{1}^{0} , e_{2}^{0} )$ of $\im \Pi_{2 3}$. Then, \eqref{a69} follows directly from Lemma \ref{a22}.
\end{proof}

\begin{proof}[Proof of Theorem \ref{a70}]
From Lemma \ref{a58} and \eqref{a56}, $( \partial_{\tau_{\nu}} Q ( 0 ) )_{\nu = 1 , 2 , 3 , 4}$ is a basis of $M_{2 \times 2} ( \R )$. Thus,
\begin{equation*}
d_{0} Q : \R^{4} \simeq T_{0} \R^{4} \longrightarrow T_{\lambda_{2} Id} M_{2 \times 2} ( \R ) \simeq M_{2 \times 2} ( \R ) ,
\end{equation*}
is an isomorphism. By the inverse function theorem, $\tau \mapsto Q ( \tau )$ is a local diffeomorphism from a neighborhood of $0$ to a neighborhood of $\lambda_{2} Id$, for $h$ small enough. Note that the neighborhoods may depend on $h$. Then, there exists $\tau ( h ) \in \R^{4}$ with $\vert \tau \vert < r$ such that
\begin{equation} \label{a71}
Q ( \tau ) = \begin{pmatrix}
\lambda_{2} & \rho \\
0 & \lambda_{2}
\end{pmatrix} ,
\end{equation}
for some $\rho ( h ) \neq 0$. Since $Q$ is the operator $P_{\tau}$ restricted to its stable eigenspace $\im \Pi_{2 3}^{\tau}$ in the basis $( e_{1}^{\tau} , e_{2}^{\tau} )$, \eqref{a71} shows that $P_{\rm Jor} : = P_{\tau}$ has a non-trivial Jordan block associated with the eigenvalue $\lambda_{2}$ and Theorem \ref{a70} follows.
\end{proof}

\section{Proof of Proposition \ref{a78}} \label{s6}

We write $\CP = \CP_{0} + \varepsilon \SB$ with
\begin{equation*}
\CP_{0} = - h^{2} \Delta + x^{2} - 2 \qquad \text{and} \qquad \SB = x_{1} h \partial_{x_{2}} - x_{2} h \partial_{x_{1}} .
\end{equation*}
The operator $\CP_{0}$ is the harmonic oscillator
\begin{equation*}
\CP_{0} = d_{f}^{*} \circ d_{f} = a_{1}^{*} \circ a_{1} + a_{2}^{*} \circ a_{2} - 2 ,
\end{equation*}
with the creation operators $a_{j}^{*} = - h \partial_{x_{j}} + x_{j}$ and annihilation operators $a_{j} = h \partial_{x_{j}} + x_{j}$. The spectrum of $\CP_{0}$ is $2 h \N$ and the eigenspace associated to $2 n h$ is the $( n + 1 )$-dimensional space
\begin{equation*}
E_{n} = \Vect \big\{ {a_{1}^{*}}^{k} {a_{2}^{*}}^{n - k} e^{- f / h} ; \ k = 0 , \ldots , n \big\} .
\end{equation*}
A direct computation gives
\begin{equation*}
[ \SB , a_{1}^{*} ] = - h a_{2}^{*} \qquad \text{and} \qquad [ \SB , a_{2}^{*} ] = h a_{1}^{*} ,
\end{equation*}
showing that $E_{n}$ is stable by $\SB$. Let $\SB_{n}$ denote the restriction of $\SB$ to $E_{n}$. Summing up the previous arguments, we deduce
\begin{equation} \label{a79}
\sigma ( \CP ) = \bigcup_{n = 0}^{+ \infty} 2 n h + \varepsilon \sigma ( \SB_{n} ) ,
\end{equation}
where $\sigma ( \SB_{n} ) \subset i \R$ since $\SB_{n}$ is anti-adjoint as $\SB$.

To get the spectral gap of $\CP$, it remains to compute $\sigma ( \SB_{1} )$. In the basis $( x_{1} e^{- f / h} , x_{2} e^{- f / h} )$ of $E_{1}$, the matrix of $\SB_{1}$ takes the form
\begin{equation*}
\SB_{1} =
\begin{pmatrix}
0 & h \\
- h & 0
\end{pmatrix} .
\end{equation*}
Thus, $\sigma ( \SB_{1} ) = \{ i h , - i h \}$ and Proposition \ref{a78} follows from \eqref{a79}.

\bibliographystyle{amsplain}
\providecommand{\MR}{\relax\ifhmode\unskip\space\fi MR }
\providecommand{\MRhref}[2]{%
  \href{http://www.ams.org/mathscinet-getitem?mr=#1}{#2}
}
\providecommand{\href}[2]{#2}


\end{document}